\documentclass[10pt]{amsart}
\usepackage{amssymb}
\usepackage{amscd}
\theoremstyle{plain}
\newtheorem{theorem}{Theorem}
\theoremstyle{definition}
\newtheorem{proposition}{Proposition}[section]
\newtheorem{corollary}[proposition]{Corollary}
\newtheorem{lemma}[proposition]{Lemma}
\newtheorem{definition}[proposition]{Definition}
\newtheorem{claim}[proposition]{Claim}
\newtheorem{question}[proposition]{Question}
\theoremstyle{remark}

\DeclareMathOperator{\acc}{acc}

\DeclareMathOperator{\otp}{otp}

\DeclareMathOperator{\Sk}{Sk}
\DeclareMathOperator{\drop}{Drop}
\DeclareMathOperator{\gap}{Gap}
\DeclareMathOperator{\fil}{Fill}

\DeclareMathOperator{\ran}{ran}

\DeclareMathOperator{\Ch}{Ch}

\DeclareMathOperator{\st}{St}
\newcommand{\sk}{\vskip.05in}
\DeclareMathOperator{\id}{id}

\DeclareMathOperator{\nacc}{nacc}
\newcommand{\restr}{\upharpoonright}

\DeclareMathOperator{\cf}{cf}

\newcommand{\subs}{\subseteq}

\numberwithin{equation}{section}
\begin{document}
\title{Successors of singular cardinals and coloring theorems~{II}}
\author{Todd Eisworth}
\address{Department of Mathematics\\
         Ohio University\\
         Athens, OH 45701}
\email{eisworth@math.ohiou.edu}
\author{Saharon Shelah}
\address{Institute of Mathematics\\
         The Hebrew University of Jerusalem\\
         Jerusalem, Israel\\}
\address{Department of Mathematics\\
         Rutgers University\\
         New Brunswick, NJ}
\email{shelah@math.huji.ac.il}
 \keywords{square-brackets partition relations, minimal walks, successor of singular cardinal}
 \subjclass{03E02}
 \thanks{The first author acknowledges support from NSF grant DMS 0506063. Research of the second author was supported by the United
   States-Israel Binational Science Foundation (Grant no. 2002323)
   The authors' collaboration was supported in part by NSF Grant DMS 0600940. This is paper 819 in the publication
 list of the second author.}
\date{\today}
\begin{abstract}
In this paper, we investigate the extent to which techniques used in~\cite{Sh:365}, \cite{535}, and \cite{nsbpr} --- developed to prove coloring
theorems at successors of singular cardinals of uncountable cofinality --- can
be extended to cover the countable cofinality case.
\end{abstract}
\maketitle
\section{Introduction}

In this paper, we tackle some of the issues left unresolved its predecessor~\cite{535} and the related~\cite{nsbpr}.
In particular, we begin the project of extending the coloring theorems found in those papers to a more
general setting --- a setting that will allow us to draw conclusions concerning successors
of singular cardinals of countable cofinality.

We remind the reader that the
 {\em square-brackets partition relation} $\kappa\rightarrow[\lambda]^\mu_\theta$ of
Erd\"{o}s, Hajnal, and Rado \cite{ehr}  asserts that for every function $F:[\kappa]^\mu\rightarrow\theta$
 (where $[\kappa]^\mu$ denotes the subsets of $\kappa$ of cardinality $\mu$),
there is a set $H\subs\kappa$ of cardinality $\lambda$ such that
\begin{equation}
\ran(F\restr[H]^\mu)\neq\theta,
\end{equation}
that is, the function $F$ omits at least one value when we restrict
it to $[H]^\mu$.

This paper investigates the extent to which {\em negations} of square-brackets
partition relations hold at the successor of a singular cardinal.  In particular,
we examine  relatives of the combinatorial statement
\begin{equation}
\label{eqn6}
\lambda\nrightarrow [\lambda]^2_\lambda,
\end{equation}
 where $\lambda$ is the successor of a singular cardinal. Our main concern
 is the situation where $\lambda=\mu^+$ for $\mu$ singular of countable cofinality; in general, we
 already know stronger results for the case where $\lambda$ is the successor of a singular of
 uncountable cofinality.
The added difficulties that arise in the work for this paper are due to some issues involving
club-guessing, and we prove some theorems in that area as well.

We also remark that Chapter~III of~\cite{cardarith} (i.e., \cite{Sh:365}) claims
something stronger than our Theorem~\ref{thm9}, but there is a problem in the proof
given there. More precisely, the comments on page~163 dealing with extending the main theorem
of that chapter to the successor of a singular of countable cofinality (Lemma 4.2(4)) are
not enough to push the proof through.  Theorems~\ref{mainthm} and~\cite{thm9} provide
a partial reclamation of this earlier work of the second author.

We now take a moment to fix our notation and lay out some results underpinning our work. In
particular, we need to discuss scales, elementary submodels, and their interaction.

\begin{definition}
Let $\mu$ be a singular cardinal. A {\em scale for $\mu$} is a pair
$(\vec{\mu},\vec{f})$ satisfying
\begin{enumerate}
\item $\vec{\mu}=\langle\mu_i:i<\cf(\mu)\rangle$ is an increasing sequence of regular cardinals
such that $\sup_{i<\cf(\mu)}\mu_i=\mu$ and $\cf(\mu)<\mu_0$.
\item $\vec{f}=\langle f_\alpha:\alpha<\mu^+\rangle$ is a sequence of functions such that
\begin{enumerate}
\item $f_\alpha\in\prod_{i<\cf(\mu)}\mu_i$.
\item If $\gamma<\delta<\beta$ then $f_\gamma<^* f_\beta$, where  the notation $f<^* g$  means that $\{i<\cf(\mu): g(i)\leq f(i)\}$ is bounded below $\cf(\mu)$.
\item If $f\in\prod_{i<\cf(\mu)}\mu_i$ then there is an $\alpha<\beta$ such that $f<^* f_\alpha$.
\end{enumerate}
\end{enumerate}
\end{definition}

Our conventions regarding elementary submodels are  standard
--- we assume that $\chi$ is a sufficiently large regular cardinal
and let $\mathfrak{A}$ denote the structure $\langle H(\chi),\in,
<_\chi\rangle$ where $H(\chi)$ is the collection of sets
hereditarily of cardinality less than $\chi$, and $<_\chi$ is some
well-order of $H(\chi)$.  The use of $<_\chi$ means that our
structure $\mathfrak{A}$ has definable Skolem functions, and we
obtain the set of {\em Skolem terms} for $\mathfrak{A}$ by closing
the collection of Skolem functions under composition.  With these
Skolem terms in hand, we can discuss Skolem hulls:

\begin{definition}
Let $B\subs H(\chi)$. Then $\Sk_{\mathfrak{A}}(B)$ denotes the
Skolem hull of $B$ in the structure $\mathfrak{A}$. More precisely,
\begin{equation*}
\Sk_{\mathfrak{A}}(B)=\{t(b_0,\dots,b_n):t\text{ a Skolem term for
$\mathfrak{A}$ and }b_0,\dots,b_n\in B\}.
\end{equation*}
\end{definition}

The set $\Sk_{\mathfrak{A}}(B)$ is an elementary substructure of
$\mathfrak{A}$, and it is the smallest such structure containing
every element of $B$.

We also make use of characteristic functions of elementary submodels.

\begin{definition}
\label{chardef} Let $\mu$ be a singular cardinal of cofinality~$\kappa$, and
let $\vec{\mu}=\langle\mu_i:i<\kappa\rangle$ be an increasing
sequence of regular cardinals cofinal in $\mu$.  If $M$ is an
elementary submodel of $\mathfrak{A}$ such that
\begin{itemize}
\item $|M|<\mu$,
\item $\langle \mu_i:i<\kappa\rangle\in M$, and
\item $\kappa+1\subs M$.
\end{itemize}
then the {\em characteristic function of $M$ on $\vec{\mu}$}
(denoted $\Ch^{\vec{\mu}}_M$) is the function with domain $\kappa$
defined by
\begin{equation*}
\Ch^{\vec{\mu}}_M(i):=
\begin{cases}
\sup(M\cap\mu_i) &\text{if $\sup(M\cap\mu_i)<\mu_i$,}\\
0  &\text{otherwise.}
\end{cases}
\end{equation*}
If $\vec{\mu}$ is clear from context, then we suppress reference to
it in the notation.
\end{definition}

In the situation of Definition~\ref{chardef}, it is clear that
$\Ch^{\vec{\mu}}_M$ is an element of the product
 $\prod_{i<\kappa}\mu_i$, and furthermore, $\Ch^{\vec{\mu}}_M(i)=\sup(M\cap\mu_i)$ for all
sufficiently large $i<\kappa$.  The following result is essentially due to Baumgartner~\cite{jb} ---
a proof can be found in the introductory section of~\cite{myhandbook}.

\begin{lemma}
\label{skolemhulllemma} Let $\mu$, $\kappa$, $\vec{\mu}$, and $M$ be
as in Definition~\ref{chardef}. If $i^*<\kappa$ and we define $N$ to
be $\Sk_{\mathfrak{A}}(M\cup\mu_{i^*})$, then
\begin{equation}
\Ch_M\restr [i^*+1,\kappa)=\Ch_N\restr [i^*+1,\kappa).
\end{equation}
\end{lemma}

We need one more easy fact about scales; a proof can be found in~\cite{nsbpr}.
In the statement of the lemma (and throughout the rest of this paper) we use the notation
``$\forall^*$'' to mean ``for all sufficiently large'' and ``$\exists^*$'' to mean ``there are unboundedly
many''.

\begin{lemma}
\label{scalelemma1}
Let $\lambda=\mu^+$ for $\mu$ singular of cofinality $\kappa$, and suppose $(\vec{\mu},\vec{f})$ is a scale
for $\mu$. Then there is a closed unbounded $C\subs\lambda$ such that the following holds for every $\beta\in C$:
\begin{equation}
(\forall^* i<\kappa)(\forall\eta<\mu_i)(\forall\mu<\mu_{i+1})(\exists^*\alpha<\beta)[f_\alpha(i)>\eta\wedge f_\alpha(i+1)>\nu.]
\end{equation}
\end{lemma}

\section{Club-Guessing}

In this section we investigate club-guessing.  The coloring theorems presented
in~\cite{Sh:365}, \cite{nsbpr}, and~\cite{535} make use of a particular type of club-guessing sequence. These
special club-guessing sequences are known to exist at successors of singular cardinals of uncountable cofinality
(we give a proof in this section, as the original proof in~\cite{Sh:365} has some minor problems), but it is still
open whether they must exist at successors of singular cardinals of countable cofinality. For this case, the
current section provides club-guessing sequences satisfying weaker conditions, and then in the sequel we
demonstrate that these sequences can be used to obtain similar coloring theorems.  We will begin with some terminology.

\begin{definition}
Let $\lambda$ be a cardinal.
\begin{enumerate}
\item A $C$-sequence for $\lambda$ is a family $\langle C_\alpha:\alpha<\lambda\rangle$ such that $C_\alpha$
is closed and unbounded in $\alpha$ for each $\alpha<\lambda$.
\item If $S$ is a stationary subset of $\lambda$, then an $S$-club system is a family $\langle C_\delta:\delta\in S^*\rangle$
such where
\begin{itemize}
\item $S^*$ is a subset of $S$ such that $S\setminus S^*$ is non-stationary, and
\item $C_\delta$ is closed and unbounded in $\delta$ for each $\delta\in S^*$.
\end{itemize}
\end{enumerate}
\end{definition}

As is clear by the above definition, there is precious little difference between calling
 $\langle e_\alpha:\alpha<\lambda\rangle$ a $C$-sequence and calling it a $\lambda$-club system ---
the two names exist for historical reasons.  The difference in terminology is worth preserving for other reasons,
however, because we will be using these objects in completely different ways --- ``$C$-sequences'' are used exclusively for constructing minimal walks, while ``$\lambda$-club systems'' are used only for club-guessing matters.  Our use of different terms makes it clear how the objects are to be used, and keeps our notation consistent with the extant literature.

The use of the set $S^*$ in the preceding definition is for technical reasons --- very often, we will take
an existing $S$-club system and modify in a way that makes sense only for ``almost all'' elements of $S$, and
we still would like to call the resulting object an $S$-club system.

\begin{definition}
Suppose $C$ is a closed unbounded subset of an ordinal $\delta$. Then
\begin{enumerate}
\item $\acc(C) = \{\alpha\in C: \alpha = \sup(C\cap \alpha)\}$, and

\medskip

\item $\nacc(C) = C\setminus \acc(C)$.
\end{enumerate}
If $\alpha\in\nacc(C)$, then we define $\gap(\alpha, C)$, {\em the gap in $C$ determined by $\alpha$}, by
\begin{equation}
\gap(\alpha, C) = (\sup (C\cap\alpha), \alpha).
\end{equation}
\end{definition}

The next definition captures some standard ideas from proofs of club-guessing; we have chosen more descriptive
names (due to Kojman~\cite{abc}) than those prevalent in~\cite{cardarith}.

\begin{definition}
Suppose $C$ and $E$ are sets of ordinals with $E\cap \sup(C)$ closed in $\sup(C)$.
We define
\begin{equation}
\drop(C, E) = \{\sup(\alpha\cap E):\alpha\in C\setminus\min(E)+1\}.
\end{equation}
Furthermore, if $C$ and $E$ are both subsets of some cardinal $\lambda$ and $\langle e_\alpha:\alpha<\lambda\rangle$
is a $C$-sequence, then for each $\alpha\in\nacc(C)\cap\acc(E)$, we define
\begin{equation}
\fil(\alpha, C, E) = \drop(e_\alpha, E)\cap\gap(\alpha, C).
\end{equation}
\end{definition}

Our notation suppresses the dependence on the parameter $\langle e_\alpha:\alpha<\lambda\rangle$ because
the precise choice of $e_\alpha$ does not make a difference at all; all that matters
is that $\fil(\alpha, C, E)$ provides us with a simple way of generating a closed unbounded
 subset of $E\cap\gap(\alpha, C)$ for $\alpha$ in $\nacc(C)\cap\acc(E)$.

In our first theorem, we characterize the existence of the  special sorts of club-guessing sequences that are crucial
to proofs given in~\cite{nsbpr} and~\cite{535}.

\begin{theorem}
\label{equiv}
Suppose $\lambda=\mu^+$ for $\mu$ a singular cardinal, and let $S$ be a stationary
subset of $\{\delta<\lambda:\cf(\delta)=\cf(\mu)\}$.  Then the following are equivalent:
\begin{enumerate}
\item There is an $S$-club system $\langle C_\delta:\delta\in S\rangle$ such that

\medskip

\begin{enumerate}
\item $|C_\delta|<\mu$ for every $\delta\in S$, and

\medskip

\item for every closed unbounded $E\subs\lambda$, there are stationarily many $\delta$ such that
for all $\tau<\mu$,
\begin{equation}
\label{eqn2.1}
\{\alpha\in\nacc(C_\delta)\cap E : \cf(\alpha)>\tau\}\text{ is unbounded in $\delta$.}
\end{equation}
\end{enumerate}
\item There is an $S$-club system $\langle C_\delta:\delta\in S\rangle$ such that
\begin{enumerate}

\medskip

\item $\sup\{|C_\delta|:\delta\in S\}<\mu$, and

\medskip

\item for every closed unbounded $E\subs\lambda$, there are stationarily many $\delta$ such that
for all $\tau<\mu$,
\begin{equation}
\{\alpha\in\nacc(C_\delta)\cap E : \cf(\alpha)>\tau\}\text{ is unbounded in $\delta$.}
\end{equation}
\end{enumerate}
\item There is an $S$-club system $\langle C_\delta:\delta\in S\rangle$ such that
\begin{enumerate}

\medskip

\item $\otp(C_\delta)=\cf(\delta)$ for every $\delta\in S$,

\medskip

\item $\langle\cf(\alpha):\alpha\in\nacc(C_\delta)\rangle$ is strictly increasing and cofinal in $\mu$, and

\medskip

\item for every closed unbounded $E\subs\lambda$, there are stationarily many $\delta\in S$ with $C_\delta\subs E$.
\end{enumerate}
\end{enumerate}
\end{theorem}
\begin{proof}
Assume $\bar{C}=\langle C_\delta:\delta\in S\rangle$ is as in (1). We claim that there is a $\theta<\mu$ such that
for every closed unbounded $E\subs\lambda$, there are stationarily many $\delta$ such that~(\ref{eqn2.1}) is
satisfied for all $\tau<\mu$ {\bf\em and } $|C_\delta|\leq\theta$.  Suppose this is not the case, and
let $\langle\theta_i:i<\cf(\mu)\rangle$ be an increasing sequence of cardinals cofinal in $\mu$.  For each~$i<\cf(\mu)$,
there is a closed unbounded $E_i\subs\lambda$ such that for all $\delta\in S$, either $|C_\delta|>\theta_i$
or~(\ref{eqn2.1}) fails for some $\tau<\mu$.  The contradiction is immediate upon
consideration of the club $E=\bigcap_{i<\cf(\mu)}E_i$.

Having established the existence of such a $\theta$, we can modify $\bar{C}$ by replacing those $C_\delta$
of cardinality greater than $\theta$ by an arbitrary club (in $\delta$) of order-type~$\cf(\delta)$, and
this gives us an $S$-club system as in~$(2)$.

The journey from (2) to (3) is an application of standard club-guessing ideas. If $E$ is club in $\lambda$, for the purpose of this
proof, let us agree to say {\em $\bar{C}$ guesses $E$ at $\delta$} if~(\ref{eqn2.1}) holds for all $\tau<\mu$.  Our first move
is to establish that if $\bar{C}$ is as in~(2), then there is a closed unbounded $E^*\subs\lambda$ such that
for every closed unbounded $E\subs\lambda$, there are stationarily many $\delta\in S$ where $\bar{C}$ guesses
$\acc(E^*)$ at~$\delta$
{\em\bf and } such that $\drop(C_\delta, E^*)\subs E$.

Suppose this fails. Choose a regular cardinal $\sigma$ such that
\begin{equation*}
\sup\{|C_\delta|:\delta\in S\}<\sigma<\mu.
\end{equation*}
By recursion on $\zeta<\sigma$ we choose clubs $E_\zeta$ of $\lambda$ as follows:

\begin{flushleft}

{\sf Case $\zeta=0$:}  $E_0=\lambda$

\medskip

{\sf Case $\zeta$ limit:}   We let $E_\zeta=\bigcap_{\xi<\zeta}E_\xi$.

\medskip

{\sf Case $\zeta = \xi+1$:}  In this case, by our assumption we know that $E_\xi$ does not enjoy the properties
required of $E^*$.  Thus, there are closed unbounded sets $E_\xi^0$ and $E_\xi^1$ such that
for all $\delta\in E_\xi^0\cap S$, if $\bar{C}$ guesses $\acc(E_\xi)$ at $\delta$, then there is an
 $\alpha\in C_\delta\setminus\min(E_\xi)$ such that $\sup(E_\xi\cap\alpha)\notin E_\xi^1$.  We now define

\begin{equation}
E_\zeta =E_{\xi+1}= \acc(E_\xi)\cap E_\xi^0\cap E_\xi^1
\end{equation}
and the construction continues.
\end{flushleft}

Now let $E=\bigcap_{\zeta<\sigma}E_\zeta$.  It is clear that $E$ is club in $\lambda$, and so by our assumption
 we can find $\delta\in S$ where $\bar{C}$ guesses $E$.  We note that $\delta\in E$, and
 therefore $\delta\in E^0_\zeta$ for all $\zeta<\sigma$.  Furthermore, $\bar{C}$ guesses $\acc(E_\zeta)$ at $\delta$
 for all $\zeta<\sigma$ because $E\subs \acc(E_\zeta)$.
 Our construction forces us to conclude that for each $\zeta<\sigma$, there is an
 $\alpha\in C_\delta\setminus\min(E_\zeta)$ such that $\sup(E_\zeta\cap\alpha)$ is not in  $E^1_\zeta$ (and therefore
 not in  $E_{\zeta+1}$  either).

We now get a contradiction using a well-known argument --- for each $\alpha\in C_\delta$ greater than $\min(E)$, the sequence
$\langle \sup(E_\zeta\cap \alpha):\zeta<\sigma\rangle$ is decreasing, and therefore eventually constant.  Thus, there
are $\gamma_\alpha<\delta$ and $\zeta_\alpha<\sigma$ such that
\begin{equation*}
\zeta_\alpha\leq\zeta<\sigma\Longrightarrow \sup(E_\zeta\cap\alpha)=\gamma_\alpha.
\end{equation*}
Since $|C_\delta|<\sigma$, we know $\zeta^*:=\sup\{\zeta_\alpha:\alpha\in C_\delta\}$ is less than $\sigma$.
We know $\bar{C}$~guesses $\acc(E_{\zeta^*})$ at $\delta$, and so there is an $\alpha\in C_\delta\setminus\min(E_{\zeta^*})$
such that
\begin{equation}
\sup(E_{\zeta^*}\cap\alpha)\notin E_{\zeta^*+1}.
\end{equation}
But $\zeta^*\geq\zeta_\alpha$,  so
\begin{equation}
\sup(E_{\zeta^*}\cap\alpha)=\gamma_\alpha=\sup(E_{\zeta^*+1}\cap \alpha)\in E_{\zeta^*+1},
\end{equation}
and we have our contradiction.

To finish the proof, let us suppose that $E^*$ is the club whose existence was just established. If
$\bar{C}$ guesses $\acc(E^*)$ at $\delta$, then we can easily build a set $D_\delta$ such that
\begin{itemize}

\medskip

\item $D_\delta\subs \acc(E^*)\cap C_\delta$,

\medskip

\item $D_\delta$ is closed and unbounded in $\delta$ with $\otp(D_\delta)=\cf(\delta)$, and

\medskip

\item $\langle\cf(\alpha):\alpha\in\nacc(D_\delta)\rangle$ is strictly increasing and cofinal in~$\mu$.

\end{itemize}
Notice that $D_\delta\subs \drop(C_\delta, E^*)$ for such $\delta$ --- this is the reason for using $\acc(E^*)$.
For all other $\delta\in S$, we can let $D_\delta$ be a subset of $\delta$ satisfying the
last two conditions above.  It is now routine to verify that $\langle D_\delta:\delta\in S\rangle$ is as required.
Since it is clear that (3) implies (1), the theorem has been established.

\end{proof}

Let us agree to call an $S$-club system a {\em nice club-guessing sequence}
if it satisfies (3) of the above theorem --- this is in concordance with notation from~\cite{cardarith},
and it
also fits in with the {\em nice pairs } defined in~\cite{nsbpr}. We will say that $S$ {\em carries a nice club-guessing sequence}
when such can be found.

Our next task is to demonstrate that nice club-guessing sequences exist when we deal with
successors of singular cardinals of uncountable cofinality.  This result actually follows from Claim~2.6
on page~127 of~\cite{cardarith}, but the proof of that claim has some problems.
The proof we give fixes these oversights, and is actually quite a bit simpler.

\begin{theorem}
\label{cleanthm}
If $\lambda=\mu^+$ for $\mu$ a singular cardinal of uncountable cofinality, then
every stationary subset of $\{\delta<\lambda:\cf(\delta)=\cf(\mu)\}$ carries a nice club-guessing
sequence.
\end{theorem}
\begin{proof}
Let $S$ be such a stationary set. By our previous work, it suffices to produce an $S$-club
system satisfying~(1) of Theorem~\ref{equiv}.  Assume by way of contradiction that no
such $S$-club system exists.

Let  $\langle C_\delta:\delta\in S\rangle$ be an $S$-club system with $\otp(C_\delta)=\cf(\delta)$,
and let $\bar{e}$ be any $C$-sequence on $\lambda$.

By recursion on $n<\omega$, we will define objects $\langle C_\delta^n:\delta<\omega\rangle$,
 $\langle\tau_\delta^n:\delta\in S\rangle$,
$\langle \epsilon^n_\delta:\delta\in S\rangle$, and $E_n$ such that
\begin{itemize}

\medskip

\item $C_\delta^n$ is closed and unbounded in $\delta$,

\medskip

\item $\tau_\delta^n$ is a regular cardinal less than $\mu$,

\medskip

\item $\epsilon_\delta^n<\delta$, and

\medskip

\item $E_n$ is closed and unbounded in $\lambda$.

\medskip
\end{itemize}
We let $\bar{C}^n$ denote $\langle C_\delta^n:\delta\in S\rangle$, and our initial set up
has $E_0=\lambda$, $\bar{C}^0=\bar{C}$, $\epsilon_\delta^0 = 0$, and $\tau_\delta^0 = 0$.

Suppose we are given $\bar{C}^n$. By our assumption, $\bar{C}^n$ does not satisfy the demands of our
theorem, and so there are clubs $E^0_n$ and $E^1_n$ such that $\bar{C}^n$ fails to guess $E^0_n$ on $E^1_n\cap S$.
This means for any $\delta\in E^1_n\cap S$, there are $\epsilon<\delta$ and  a regular $\tau<\mu$ such that
\begin{equation}
\alpha\in\nacc(C_\delta^n)\cap E^1_0 \Longrightarrow \cf(\alpha)\leq\tau.
\end{equation}
We now define $E_{n+1}= \acc(E_n\cap E_n^0\cap E_n^1)$, define $\epsilon_\delta^{n+1}$ to be the
least such $\epsilon$, and define $\tau_\delta^{n+1}$ to be the least $\tau$ corresponding to $\epsilon^{n+1}_\delta$.

Now that $E_{n+1}$ has been defined, we declare an ordinal $\delta\in S$ to be {\em active at stage $n+1$}
if $\delta\in \acc(E_{n+1})$.  For those $\delta\in S$ that are inactive at stage $n+1$, we do nothing ---
set $C_\delta^{n+1}= C_\delta^n$,
$\tau_\delta^{n+1}=\tau_\delta^n$, and $\epsilon_\delta^{n+1}=\epsilon_\delta^n$.

For the remainder of this construction, we $\delta$ {\em is} active at stage $n+1$.
Let us say that ordinal $\alpha<\delta$ {\em needs attention at stage $n+1$} if
\begin{equation}
\label{eqn22}
\alpha\in \nacc(C_\delta^n)\cap \acc(E_{n+1})\setminus \epsilon_\delta^{n+1}+1.
\end{equation}
Notice that any ordinal requiring attention at this stage is necessarily of cofinality at most $\tau^{n+1}_\delta$.

Our construction $C^{n+1}_\delta$ commences by setting
\begin{equation}
D_\delta^n = \drop(C_\delta, E_{n+1}).
\end{equation}
This set $D_\delta^n$ is still closed and unbounded in $\delta$ since $\delta$ is active, and if
 $\alpha$ needed attention at this stage, then $\alpha=\sup(E_{n+1}\cap\alpha)$ and therefore
\begin{equation}
\label{eqn20}
\alpha\in\nacc(D^n_\delta)\cap \acc(E_{n+1}).
\end{equation}
In particular, the set $\fil(\alpha, D^n_\delta, E_{n+1})$ is defined for any $\alpha$ that needs attention at this stage.

To finish the construction, we define
\begin{equation}
C^{n+1}_\delta = D^n_\delta\cup\{\fil(\alpha, D^n_\delta[E], E_{n+1}): \text{ $\alpha$ needs attention }\}.
\end{equation}
The set $C^{n+1}_\delta$ is clearly unbounded in $\delta$, and it is closed since it was obtained from
$D^n_\delta$ by gluing closed sets into ``gaps'' in $D^n_\delta$.  It remains to see that $|C^{n+1}_\delta|<\mu$,
and this follows by the estimate
\begin{equation}
\left|C^{n+1}_\delta\right|\leq \left|C^n_\delta\right| + \tau_\delta^{n+1}\cdot\left|C^n_\delta\right|.
\end{equation}
Thus, the recursion can continue.

Let $E = \bigcap_{n<\omega} E_n$, and choose $\delta\in S\cap \acc(E)$ such that $\mu$ divides
the order-type of $\delta\cap E$. Since $E\subs\acc(E_n)$ for all $n$, it follows
that $\delta$ is active at all stages of the construction.  Let us define
\begin{equation}
\epsilon^* = \sup\{\epsilon_\delta^n:n<\omega\}+1,
\end{equation}
and
\begin{equation}
\label{eqn21}
\theta^* = \sup\{|C^n_\delta|:n<\omega\}.
\end{equation}
Since $\aleph_0<\cf(\mu)=\cf(\delta)$, we know $\epsilon^*<\delta$ and $\theta^*<\mu$.  Since $\delta\in\acc(E)$
and $\mu$ divides $\otp(E\cap\delta)$,
\begin{equation}
|E\cap \delta\setminus \epsilon^*| = \mu,
\end{equation}
and an appeal to~(\ref{eqn21}) tells us that we can choose an ordinal $\gamma$ such that
\begin{itemize}

\medskip

\item $\gamma\in E$

\medskip

\item $\epsilon^* < \gamma < \delta$, and

\medskip

\item $\gamma\notin\bigcup_{n<\omega}C^n_\delta$.

\medskip
\end{itemize}
Our next move involves consideration of the sequence $\langle \alpha_n:n<\omega\rangle$ of ordinals defined as
\begin{equation}
\alpha_n = \min(C^n_\delta\setminus\gamma).
\end{equation}
We will reach a contradiction by proving that this sequence of ordinals is strictly decreasing.

Note that $\alpha_n$ is necessarily greater than $\gamma$ by our choice of $\gamma$. This means
that $\alpha_n$ is an element of $\nacc(C^n_\delta)$. Moreover,
\begin{equation}
\epsilon^{n+1}_\delta<\epsilon^*\leq \alpha_n.
\end{equation}
Two possibilities now arise --- either $\alpha_n$ needs attention at stage $n+1$, or it does not.
We analyze each of these cases individually.

\medskip

\noindent{\sf Case 1:  $\alpha_n$ does not need attention at stage $n+1$}

\medskip

A glance at~(\ref{eqn22}) establishes that $\alpha_n$ is not an element of $\acc(E_{n+1})$, and
hence if we set $\beta_n  = \sup(\alpha_n\cap E_{n+1})$, then $\beta_n<\alpha_n$.
Now $\gamma\in E\subs E_{n+1}$, and therefore.
\begin{equation}
\gamma\leq\beta_n <\alpha_n.
\end{equation}
The ordinal $\beta_n$ is in $D^n_\delta$ which is itself a subset of $C^{n+1}_\delta$ and so
\begin{equation}
\alpha_{n+1}\leq\beta_n<\alpha_n.
\end{equation}

\medskip

\noindent{\sf Case 2: $\alpha_n$ needs attention at stage $n+1$}

\medskip

In this case, we have seen that $\fil(\alpha_n, D^n_\delta, E_{n+1})$ is closed and unbounded in~$\alpha_n$
and included in $C^{n+1}_\delta$.  Since $\gamma$ must be strictly less than $\alpha_n$, we see
\begin{equation}
\gamma<\alpha_{n+1}\leq\min(\fil(\alpha_n, D^n_\delta, E_{n+1})\setminus\gamma)<\alpha_n
\end{equation}
and again we have $\alpha_{n+1}<\alpha_n$.

\medskip

We now have the desired contradiction, as $\langle \alpha_n:n<\omega\rangle$ allegedly forms a strictly
decreasing sequence of ordinals.
\end{proof}

  We now come to a very natural question that is still open.

\begin{question}
\label{q1}
Suppose $\lambda=\mu^+$ for $\mu$ singular of countable cofinality, and let $S$ be
a stationary subset of $\{\delta<\lambda:\cf(\mu)=\omega\}$. Does $S$ carry a nice club-guessing sequence?
\end{question}

This question is particular relevant for this paper because a positive answer would allow us to strengthen
our results, as well as simplify the proof enormously by using the techniques of~\cite{nsbpr}. A positive
answer follows easily from $\diamondsuit(S)$, but we leave the proof of this to the reader. The next
theorem explores the extent to which we can obtain $S$-club systems with properties that approximate ``niceness''.

\begin{theorem}
\label{thm2}
Let $\lambda=\mu^+$ for $\mu$ a singular cardinal of countable cofinality, and
let $S$ be a stationary subset of $\{\delta<\lambda:\cf(\delta)=\aleph_0\}$.
Further suppose that we have sequences $\langle c_\delta:\delta\in S\rangle$ and
$\langle f_\delta:\delta\in S\rangle$ such that
\begin{enumerate}

\medskip

\item $c_\delta$ is an increasing function from $\omega$ onto a cofinal subset of $\delta$ (for convenience,
 we define $c_\delta(-1)$ to be$-1$)

\medskip

\item $f_\delta$ maps $\omega$ to the set of regular cardinals less than $\mu$, and

\medskip

\item for every closed unbounded $E\subs\lambda$, there are stationarily many $\delta\in S$ such that $c_\delta(n)\in E$ for all $n<\omega$.

\medskip

\end{enumerate}
Then there is an $S$-club system $\langle C_\delta:\delta\in S\rangle$ such that
\begin{enumerate}
\setcounter{enumi}{3}

\medskip

\item $c_\delta(n)\in C_\delta$ for all $n$,

\medskip

\item $|C_\delta\cap (c_\delta(n-1),c_\delta(n)]|\leq f_\delta(n)$, and

\medskip

\item for every closed unbounded $E\subs\lambda$, there are stationarily many $\delta\in S$ such that
\begin{equation}
(\forall n<\omega)(\exists\alpha\in\nacc(C_\delta)\cap
E)\left[c_\delta(n-1)<\alpha<c_\delta(n)\text{ and
}\cf(\alpha)>f_\delta(n)\right]
\end{equation}
\end{enumerate}
\end{theorem}

We can get a picture of the case of most interest to us in the
following manner. First, notice that the functions $\langle
c_\delta:\delta\in S\rangle$ are essentially a ``standard''
club-guessing sequence of the sort we know exist.
Given $\delta\in S$, the sequence $c_\delta$ chops $\delta$
into an $\omega$ sequence of half-open intervals of the form $(c_\delta(n-1), c_\delta(n)]$.
If we define
\begin{equation}
I_\delta(n):=(c_\delta(n-1), c_\delta(n)],
\end{equation}
then $C_\delta$ is constructed so that $C_\delta\cap I_\delta(n)$ is of cardinality {\em at most}~$f_\delta(n)$.
The club-guessing property tells us that for any closed unbounded $E\subs\lambda$, there are stationarily
many $\delta\in S$ such that for each $n<\omega$, $E\cap\nacc(C_\delta)\cap I_\delta(n)$ contains
an ordinal of cofinality {\em greater than}~$f_\delta(n)$.  In particular, if the sequence $\langle f_\delta(n):n<\omega\rangle$
increases to $\mu$ for all $\delta\in S$, then for every closed unbounded $E\subs\lambda$ there
are stationarily many $\delta\in S$ such that for any $\tau<\mu$,
\begin{equation}
\{\alpha\in E\cap\nacc(C_\delta):\cf(\alpha)>\tau\}\text{ is unbounded in }\delta.
\end{equation}
This almost gives us the assumptions needed to apply Theorem~\ref{equiv}; the problem, however,
is that our hypotheses admit the possibility that~$C_\delta$ is of cardinality $\mu$, and
 this takes us out of the purview of Theorem~\ref{equiv}.

\begin{proof}
Our starting point for this proof is the bare-bones sketch of a similar proof given for Claim~2.8 on page~131 of~\cite{cardarith}.
By way of contradiction, assume that there is no such family $\langle C_\delta:\delta\in S\rangle$.
The proof will require us to construct many $S$-club systems in an attempt to produce the desired object;
let us agree to say that an $S$-club system {\em satisfies the structural requirements of Theorem~\ref{thm2}}
if conditions (4) and (5) hold, and say it {\em satisfies the club-guessing requirements of Theorem~\ref{thm2}}
if condition~(6) holds.

The main thrust of our construction is to define objects $E_\zeta$ and
 $\bar{C}^\zeta=\langle C^\zeta_\delta:\delta\in S\rangle$ by induction on $\zeta<\omega_1$.
 The sets $E_\zeta$ will be closed unbounded
 in $\lambda$, while each $\bar{C}^\zeta$ will be an $S$-club system satisfying the structural requirements
 of Theorem~\ref{thm2}.  Our convention is that {\em stage $\zeta$} refers to the process of defining
 $\bar{C}^{\zeta+1}$ and $E_{\zeta+1}$ from $\bar{C}^\zeta$ and $E_\zeta$.  The reader should also be
 warned that several auxiliary objects will be defined along the way.

\medskip

\noindent{\bf Construction}

\bigskip

\noindent{\sf Initial set-up}

\medskip

We set $E_0 = \lambda$ and $C_\delta^0 = \{c_\delta(n):n<\omega\}$
for each $\delta\in S$.

\bigskip

\noindent{\sf Stage $\zeta$ --- defining $E_{\zeta+1}$ and $\bar{C}^{\zeta+1}$}

\medskip

We assume that $\bar{C}^{\zeta}$ is an $S$-club system satisfying the structural requirements of Theorem~\ref{thm2},
and $E_\zeta$ is a closed unbounded subset of $\lambda$.  We have assumed that Theorem~\ref{thm2} fails,
and so there are closed unbounded subsets $E^0_\zeta$ and $E^1_\zeta$ of $\lambda$
 such that for each $\delta\in E^0_\zeta\cap S$, there is an $n<\omega$ such that
\begin{equation}
\alpha\in\nacc(C_\delta)\cap E^1_\zeta\cap I_\delta(n)\Longrightarrow \cf(\alpha)\leq f_\delta(n).
\end{equation}
We define
\begin{equation}
E_{\zeta+1}:=\acc(E_\zeta\cap E_\zeta^0\cap E_\zeta^1).
\end{equation}

Let us agree to say that an ordinal $\delta\in S$ is {\em active at stage $\zeta$} if $C_\delta^0\subs \acc(E_{\zeta+1})$,
and note that the set of such $\delta$ is stationary.  If $\delta\in S$ is inactive at stage~$\zeta$,
then we do nothing and let $C^{\zeta+1}_\delta= C^\zeta_\delta$.

If $\delta$ {\em is } active at stage $\zeta$, then we know $\delta\in E^0_\zeta$ and so there
is a least $n(\delta,\zeta)<\omega$ such that
\begin{equation}
\alpha\in \nacc(C_\delta^\zeta)\cap E^1_\zeta\cap I_\delta(n(\delta,\zeta))\Longrightarrow \cf(\alpha)\leq f_\delta(n(\delta,\zeta)).
\end{equation}
The construction of $C^{\zeta+1}_\delta$ will modify $C_\delta^\zeta$ only on the interval
 $I_\delta(n(\delta,\zeta))$, that is, we ensure that
\begin{equation}
C^{\zeta+1}_\delta\cap (\delta\setminus I_\delta(n(\delta,\zeta)))= C^\zeta_\delta\cap (\delta\setminus I_\delta(n(\delta,\zeta))).
\end{equation}
Our next step is to define
\begin{equation}
D^\zeta_\delta = \drop(C^\zeta_\delta\cap I_\delta(n(\delta,\zeta)), E_{\zeta+1}\cap I_\delta(n(\delta,\zeta))).
\end{equation}
Note that $D^\zeta_\delta$ is a closed unbounded subset of $c_\delta(n(\delta,\zeta))$ of cardinality
at most $f_\delta(n(\delta,\zeta))$.

We still have a some distance to traverse before arriving at $C^{\zeta+1}_\delta$ --- one should think
of $D^\zeta_\delta$ as being the first approximation to how $C^{\zeta+1}_\delta$ will look on the
interval $I_\delta(n(\delta,\zeta))$.  To finish, let us say that an element $\alpha$ of $D^\zeta_\delta$
{\em needs attention} if
\begin{itemize}
\item $\alpha\in\acc(E_{\zeta+1})\cap\nacc(D^\zeta_\delta)$, and
\medskip
\item $\cf(\alpha)\leq f_\delta(n(\delta,\zeta))$.
\end{itemize}

If $\alpha$ needs attention, then
 $\fil(\alpha,C^\zeta_\delta\cap I_\delta(n(\delta,\zeta)), E_{\zeta+1}\cap I_\delta(n(\delta,\zeta))$
is closed and unbounded in $\gap(\alpha, C^\zeta_\delta)$ and of cardinality $\cf(\alpha)\leq f_\delta(n(\delta, \zeta))$.
We define
 \begin{equation}
 A_\delta^\zeta = D^\zeta_\delta \cup\{\fil(\alpha,C^\zeta_\delta\cap I_\delta(n(\delta,\zeta)), E_{\zeta+1}\cap I_\delta(n(\delta,\zeta)):\alpha\text{ needs attention }\}.
 \end{equation}

Since the needed instances of ``$\fil$'' are always a closed subsets lying in a ``gap'' of $D^\zeta_\delta$,
 the set $A^\zeta_\delta$ is still closed and unbounded in $c_\delta(n(\delta, \zeta))$. Also, simple cardinality
 estimates tell us
\begin{equation}
\left|A^\zeta_\delta\right|\leq f_\delta(n(\delta,\zeta)).
\end{equation}
We now define $C^{\zeta+1}_\delta$ piecewise:
\begin{equation}
C^{\zeta+1}\cap\delta\setminus I_\delta(n(\delta,\zeta)) = C^\zeta_\delta,
\end{equation}
and
\begin{equation}
C^{\zeta+1}\cap I_\delta(n(\delta,\zeta)) = A^\zeta_\delta.
\end{equation}
So defined, our $S$-club system $\bar{C}^{\zeta+1}$ satisfies the structural requirements of Theorem~\ref{thm2}
and the construction continues.

\bigskip

\noindent{\sf $\bar{C}^\zeta$ and $E_\zeta$ for $\zeta$ limit}

\medskip

We begin by setting $E_\zeta = \bigcap_{\xi<\zeta}E_\xi$.  Next, for each $\delta\in S$
we let $C^\zeta_\delta$ be the closure\footnote{This is actually not necessary, as it can be
shown that the set defined in~(\ref{eqn13}) is closed.}  in $\delta$ of
\begin{equation}
\label{eqn13}
\{\alpha:\alpha\in C^\xi_\delta\text{ for all sufficiently large }\xi<\zeta\}.
\end{equation}
The set $C_\delta^\zeta$ defined above is closed in $\delta$ by definition. Since it contains
$C_\delta^0$, it is also unbounded.  Finally,
\begin{equation}
C_\delta^\zeta\cap I_\delta(n)\subs \bigcup_{\xi<\zeta}C_\delta^\xi\cap I_\delta(n).
\end{equation}
Since $\zeta$ is countable and $f_\delta(n)$ is a cardinal, it follows that
\begin{equation}
\left|C^\zeta_\delta\cap I_\delta(n)\right|\leq f_\delta(n)
\end{equation}
for all $n$, and therefore $\langle C^\zeta_\delta:\delta\in S\rangle$ satisfies the structural
requirements of Theorem~\ref{thm2}.

\medskip

\noindent{\bf End Construction}

\medskip

Having constructed $\bar{C}^\zeta$ and $E_\zeta$ for all $\zeta<\omega_1$, we turn now to obtaining a contradiction.
Let us define
\begin{equation}
E^*:=\bigcap_{\zeta<\omega_1}E_\zeta.
\end{equation}
It is clear that $E^*$ is club in $\lambda$, and so there is a $\delta\in S$ such that
\begin{equation}
\label{eqn14}
C_\delta^0\subs \{\alpha<\lambda: \mu\text{ divides }\otp(E^*\cap\alpha)\}.
\end{equation}
Let us fix such a $\delta$, and note that
\begin{equation}
\label{2.27}
|E^*\cap I_\delta(n)|=\mu \text{ for all }n<\omega.
\end{equation}

For each $\zeta<\omega_1$, we know from (\ref{eqn14}) that $\delta$ is active at each stage $\zeta<\omega_1$.
In particular, $n(\delta, \zeta)$ is defined for all $\zeta<\omega_1$ and hence there is a least $n^*<\omega$
such that $n(\delta, \zeta)=n^*$ for infinitely many $\zeta$.  Let $\langle \zeta_n:n<\omega\rangle$
list the first $\omega$ such ordinals, and let $\zeta^*=\sup\{\zeta_n:n<\omega\}$.

Choose an ordinal $\beta^*\in E^*\cap I_\delta(n^*)\setminus \bigcup_{\xi<\zeta^*} C^\xi_\delta$ ---
this is possible because of~(\ref{2.27}), as
\begin{equation}
\left|\bigcup_{\xi<\zeta^*} C^\xi_\delta\cap I_\delta(n^*)\right|\leq \aleph_0\cdot f_\delta(n^*)<\mu.
\end{equation}
Finally define
\begin{equation}
\beta_n := \min(C^{\zeta_n}_\delta\setminus\beta^*)
\end{equation}
for each $n<\omega$.  Notice that our choice of $\beta^*$ guarantees that $\beta^*$ is strictly
less than $\beta_n$ for all $n$.

\begin{claim}
For each $n$, we have $\beta_{n+1}<\beta_n$.
\end{claim}
\begin{proof}
Fix $n$.  It is clear from our construction that
\begin{equation}
\min(C^{\zeta_n+1}_\delta\setminus\beta^*)=\min(C^{\zeta_{n+1}}_\delta)=\beta_{n+1}
\end{equation}
because $\beta_n\in I_\delta(n^*)$ and $n(\delta,\xi)\neq n^*$ if $\zeta_n<\xi<\zeta_{n+1}$.

We now track what happens to $\beta_n$ during stage $\zeta_n$ by splitting into two cases.

\bigskip

\noindent{\sf Case 1: $\beta_n \notin \acc(E_{\zeta_n+1})$.}

\medskip

In this case, we note that since $\beta^*\in E_{\zeta_n+1}$ we have
\begin{equation}
\beta^*\leq \sup(\beta_n\cap E_{\zeta_n+1})<\beta_n.
\end{equation}
Since $\beta^*\notin C^{\zeta_n+1}$ while
\begin{equation}
\sup(\beta_n\cap E_{\zeta_n+1})\in D^{\zeta_n+1}_\delta\subs C^{\zeta_n+1}_\delta,
\end{equation}
it follows that
$\beta^*<\beta_{n+1}<\beta_n$ and we are done.

\bigskip

\noindent{\sf Case 2: $\beta_n\in \acc(E_{\zeta_n+1})$.}

\medskip

Since $\beta^*<\beta_n$, the definition of $\beta_n$ tells us that $\beta_n$ must be in
$\nacc(C^{\zeta_n}_\delta)$.
Also, both~$\delta$ and $\beta_n$ are in $E_{\zeta_n+1}$, so in particular $\delta\in E^0_{\zeta_n}$
and $\beta_n\in E^1_{\zeta_n}$.  This tells us $\cf(\beta_n)\leq f_\delta(n(\delta,\zeta_n))$.

By our case hypothesis, $\beta_n=\sup(E_{\zeta_n+1}\cap\beta_n)$ and so $\beta_n \in D^{\zeta_n}_\delta$
and
\begin{equation}
\beta_n = \min(D^{\zeta_n}_\delta\setminus\beta^*)>\beta^*.
\end{equation}
We conclude
\begin{equation}
\beta_n\in \nacc(D^{\zeta_n}),
\end{equation}
and so $\beta_n$ needs attention during the construction of $C^{\zeta_n+1}_\delta$.
In particular,

\begin{equation}
\fil\left(\beta_n,C^\zeta_\delta\cap I_\delta\bigl(n(\delta,\zeta)\bigr), E_{\zeta+1}\cap I_\delta\bigl(n(\delta,\zeta)\bigr)\right)\subs C^{\zeta_n+1}_\delta
\end{equation}
and so
\begin{equation}
C_\delta^{\zeta_n+1}\cap (\beta^*,\beta_n)\neq\emptyset.
\end{equation}
We conclude
\begin{equation}
\beta^*<\beta_{n+1}=\min(C^{\zeta_n+1}_\delta\setminus\beta^*)=\min(C^{\zeta_{n+1}}_\delta\setminus\beta^*)<\beta_n
\end{equation}
as required.
\end{proof}
Using the preceding claim, we get a strictly decreasing set of ordinals.
This is absurd, and Theorem~\ref{thm2} is established.
 \end{proof}

Club-guessing systems structured like those provided by Theorem~\ref{thm2} will occupy our attention
for the rest of this paper, so we will give them a name.

\begin{definition}
\label{wellformeddef} Let $\lambda=\mu^+$ for $\mu$ singular of
countable cofinality, and let $S$ be a stationary subset of
$\{\delta<\lambda:\cf(\delta)=\aleph_0\}$.  An $S$-club system
$\langle C_\delta:\delta\in S\rangle$ is {\em well-formed} if there
is a function $f_{\bar{C}}:\omega\rightarrow \mu$ and functions
 $c_\delta:\omega\rightarrow\delta$ for each $\delta\in S$ such that
such that
\begin{enumerate}
\item $c_\delta$ is strictly increasing with range cofinal in $\delta$
\medskip
\item $\langle f_{\bar{C}}(n):n<\omega\rangle$ is a strictly increasing sequence of regular cardinals cofinal in $\mu$
\medskip
\item for each $n$, $|C_\delta \cap (c_\delta(n-1), c_\delta(n)]|\leq f_{\bar{C}}(n)$
\medskip
\item for each $n$, if $\alpha\in \nacc(C_\delta\cap (c_\delta(n-1),c_\delta(n)]$ then $\cf(\alpha)>f_{\bar{C}}(n)$
\medskip
\item if $E$ is closed and unbounded in $\lambda$, then there are stationarily many $\delta\in S$ such that
\begin{equation}
E\cap\nacc(C_\delta)\cap (c_\delta(n-1), c_\delta(n)]\neq\emptyset\text{ for all }n<\omega.
\end{equation}
\end{enumerate}
If there is a well-formed $S$-club system, then we say that {\em $S$ carries a well-formed club-guessing
sequence}.  We continue to use the notation $I_\delta(n)$ to indicate the interval
 $(c_\delta(n-1),c_\delta(n)]$ (where our convention is that $c(-1)=-1$), and refer to this sequence
 of intervals as {\em the interval structure of $C_\delta$}.  The function $f_{\bar{C}}$ is said to {\em measure $\bar{C}$}.
\end{definition}

\begin{proposition}
Let $S$ be a stationary subset of $\{\delta<\lambda:\cf(\delta)=\aleph_0\}$ where $\lambda=\mu^+$ with $\mu$ singular
of countable cofinality.  If $f:\omega\rightarrow\mu$ enumerates a strictly increasing sequence of regular cardinals
that is cofinal in $\mu$, then $S$ carries a well-formed club-guessing sequence that is measured by $f$.
\end{proposition}
\begin{proof}
For each $\delta\in S$, we set  $f_\delta = f$ and apply Theorem~\ref{thm2}.  The $S$-club system
 $\langle C_\delta:\delta\in S\rangle$ that arises need not satisfy condition (4) of
Definition~\ref{wellformeddef}, so for  each $\delta\in S$ we define
\begin{equation}
D_\delta^*=\{\alpha\in\nacc(C_\delta): \text{ if $\alpha\in I_\delta(n)$, then $\cf(\alpha)>f_\delta(n)$}\},
\end{equation}
and let $D_\delta$ equal the closure of $D_\delta^*$ in $\delta$.  The proof that $\langle D_\delta:\delta\in S\rangle$
is as required is routine and left to the reader.
\end{proof}

We remark that any $S$-club system $\langle C_\delta:\delta\in S\rangle$ providing a positive answer to Question~\ref{q1}
is also  essentially well-formed ---  given
any increasing function $f$ mapping $\omega$ onto a set of regular cardinals cofinal in $\mu$, it is straightforward
to ``thin out'' the $C_\delta$ to get a well-formed $S$-club system $\bar{D}$ measured by $f$.

We move now to some terminology concerning club-guessing ideals taken from~\cite{cardarith}.  We start
with a basic definition.

\begin{definition}
Let $\bar{C}=\langle C_\delta:\delta\in S\rangle$ be an $S$-club system for $S$ a stationary subset of
some cardinal~$\lambda$, and suppose $\bar{I}=\langle I_\delta:\delta\in S\rangle$ is a sequence such that
$I_\delta$ is an ideal on $C_\delta$ for each $\delta\in S$. The ideal $\id_p(\bar{C},\bar{I})$
consists of all sets $A\subs\lambda$ such that for some closed unbounded $E\subs\lambda$,
\begin{equation}
\delta\in S\cap E\Longrightarrow E\cap A\cap C_\delta\in I_\delta.
\end{equation}
\end{definition}

\begin{proposition}
\label{idpprop}
Suppose $\lambda=\mu^+$ for $\mu$ singular of countable cofinality, and let $\bar{C}$ be a well-formed
$S$-club system for some stationary $S\subs\{\delta<\lambda:\cf(\delta)=\aleph_0\}$.  Let $I_\delta$
be the ideal on $C_\delta$ generated by sets of the form
\begin{equation}
\{\gamma\in C_\delta:\gamma\in\acc(C_\delta)\text{ or }\cf(\gamma)<\alpha\text{ or }\gamma<\beta\}
\end{equation}
for $\alpha<\mu$ and $\beta<\delta$.  Then $\id_p(\bar{C},\bar{I})$ is a proper ideal.
\end{proposition}
\begin{proof}
We need to verify that $\lambda\notin \id_p(\bar{C},\bar{I})$.  If we unpack the meaning of this,
we see that we need that for every closed unbounded $E\subs\lambda$, there is a $\delta\in S$
such that $E\cap C_\delta\notin I_\delta$.  This means that for each $\alpha<\mu$ and $\beta<\delta$,
there needs to be a $\gamma\in E\cap\nacc(C_\delta)$ greater than $\beta$ with cofinality greater than~$\alpha$,
and this follows immediately from the definition of well-formed.
\end{proof}

With the preceding proposition in mind, if we say that $(\bar{C},\bar{I})$ is a well-formed $S$-club
system, we mean that $\bar{C}$ is as in Definition~\ref{wellformeddef}, and $\bar{I}=\langle I_\delta:\delta\in S\rangle$
is the sequence of ideals defined as in Proposition~\ref{idpprop}.  The ideals $\id_p(\bar{C},\bar{I})$
for well-formed $(\bar{C},\bar{I})$
lie at the heart of the coloring theorems presented in the sequel.

\section{Parameterized Walks}

In this section, we develop a generalization of Todor{\v{c}}evi{\'c}'s technique of minimal walks~\cite{stevochapter,minimal,stevobook}.
The notation is a bit cumbersome, but this seems to be unavoidable given the complexity of the ideas we are trying to voice.

\begin{definition}
Let $\lambda$ be a cardinal.  A {\em generalized $C$-sequence} is a family $$\langle e_\alpha^n:\alpha<\lambda,n<\omega\rangle$$
such that for each $\alpha<\lambda$ and $n<\omega$,
\begin{itemize}
\item $e^n_\alpha$ is closed unbounded in $\alpha$, and
\medskip
\item $e^n_\alpha\subs e^{n+1}_\alpha$.
\end{itemize}
\end{definition}

The next lemma connects the above definition with concepts from the preceding section.

\begin{lemma}
Let $\lambda=\mu^+$ for $\mu$ singular of countable cofinality, and let $(\bar{C},\bar{I})$ be a well-formed
$S$-club system for some stationary $S\subs\lambda$ consisting of ordinals of countable cofinality.  There
is a generalized $C$-sequence $\langle e^n_\alpha:\alpha<\lambda, n<\omega\rangle$ such that
\begin{itemize}
\item $|e^n_\alpha|\leq \cf(\alpha)+f_{\bar{C}}(n)+\aleph_1$, and
\medskip
\item $\delta\in S\cap e^n_\alpha\Longrightarrow C_\delta\cap I_\delta(n)\subs e^n_\alpha$.
\end{itemize}
\end{lemma}
\begin{proof}
We will obtain $e^n_\alpha$ as the closure (in $\alpha$) of a union
of approximations $e^n_\alpha[\beta]$ for $\beta<\omega_1$. We start by letting $e_\alpha$ be
closed unbounded in $\alpha$ of order-type $\cf(\alpha)$ for each $\alpha<\lambda$. The
construction proceeds as follows:
\begin{eqnarray*}
e^0_\alpha[0]&= &e_\alpha\\
e^n_\alpha[\beta+1] &= &\text{ closure in $\alpha$ of }e^n_\alpha[\beta]\cup \bigcup_{\delta\in S\cap e^n_\alpha[\beta]}C_\delta\cap I_\delta(n)\\
e^{n+1}_\alpha[0] & = &e^n_\alpha\\
e^n_\alpha[\beta] &= &\text{ closure in $\alpha$ of
}\bigcup_{\gamma<\beta}e^n_\alpha[\gamma]\text{ for $\beta$ limit}\\
e^n_\alpha &= &\text{ closure in $\alpha$ of }\bigcup_{\beta<\omega_1}e^n_\alpha[\beta].
\end{eqnarray*}
The verification that $\langle e^n_\alpha:\alpha<\lambda,n<\omega\rangle$ has the required properties
is routine.
\end{proof}

The relationship between the generalized $C$-sequence obtained above and the given well-formed
$S$-club system $(\bar{C},\bar{I})$ is important enough that it ought to have a name.

\begin{definition}
\label{swallow}
Let $\lambda=\mu^+$ for $\mu$ singular of cofinality $\aleph_0$, and
suppose $(\bar{C},\bar{I})$ is a well-formed $S$-club system for
some stationary $S\subs\{\delta<\lambda:\cf(\delta)=\aleph_0\}$. A
generalized $C$-sequence $\bar{e}$ is said to {\em swallow
}$(\bar{C},\bar{I})$ if
\begin{enumerate}
\item $|e^n_\alpha|\leq \cf(\alpha)+f_{\bar{C}}(n)+\aleph_1$, and
\medskip
\item $\delta\in S\cap e^n_\alpha\Longrightarrow C_\delta\cap I_\delta(n)\subs e^n_\alpha$.
\end{enumerate}
\end{definition}

The most important property enjoyed by these cumbersome generalized $C$-sequences is isolated
by the following lemma.

\begin{lemma}
\label{swallowlemma}
Suppose $\bar{e}$ swallows the well-formed $S$-club system
$(\bar{C},\bar{I})$. If $\delta$ is in  $S\cap e^m_\alpha$ for some
$m<\omega$, then
\begin{equation}
(\forall^*n<\omega)\left[\nacc(C_\delta)\cap I_\delta(n)\subs \nacc(e^n_\alpha)\right].
\end{equation}
\end{lemma}
\begin{proof}
Choose $n^*<\omega$ so large that $m<n^*$ and $\cf(\alpha)\leq f_{\bar{C}}(n^*)$.
If $n^*\leq n<\omega$ and $\gamma\in\nacc(C_\delta)\cap I_\delta(n)$, then
$\gamma\in e^n_\alpha$ by Definition~\ref{swallow}, and $\gamma$ cannot
be in $\acc(e^n_\alpha)$ because
\begin{equation}
|e^n_\alpha|\leq\cf(\alpha)+f_{\bar{C}}(n)+\aleph_1<\cf(\gamma).
\end{equation}
\end{proof}

Up until this point in the section, we have been developing the context in which our generalized minimal
walks will take place, and now we turn to their definition.

\begin{definition}
\label{defn1}
Let $\bar{e}$ be a generalized $C$-sequence on some cardinal $\lambda$, and let $s$ be a finite
sequence of natural numbers.  Given $\alpha<\beta<\lambda$, we define
We define $\st(\alpha,\beta,s,\ell)$ --- ``step
$\ell$ on the $s$-walk from $\beta$ to $\alpha$ (along
$\bar{e}$)'' --- by induction on $\ell<\omega$.
\begin{equation*}
\st(\alpha,\beta,s, 0)=\beta,
\end{equation*}
and
\begin{equation*}
\st(\alpha,\beta,s,\ell+1)=
\begin{cases}
\alpha &\text{if $\alpha=\st(\alpha,\beta,s,\ell)$}\\
&\\
\min(e^0_{\st(\alpha,\beta,s,\ell)}\setminus\alpha)&\text{if $\st(\alpha,\beta,s,\ell)>\alpha$ and $\ell\geq\lg(s)$}\\
&\\
\min(e^{s(\ell)}_{\st(\alpha,\beta,s,\ell)}\setminus\alpha) &\text{otherwise.}
\end{cases}
\end{equation*}
Finally, let
\begin{equation*}
n(\alpha,\beta,s)=\text{least $\ell$ such that $\alpha=\st(\alpha,\beta,s,\ell)$.}
\end{equation*}
\end{definition}

In the $C$-sequences used by Todor{\v{c}}evi{\'c}, at each stage of a minimal walk one has a single ladder
to use to make the next step. In our context, there are infinitely many ladders available, and the
parameter $s$ selects the one we use for our next step. Even though there are infinitely many ladders
available, nevertheless there are only finitely many possible destinations --- for given $\alpha<\beta$,
the sequence $\langle e^n_\beta:n<\omega\rangle$ increases with $n$ and therefore the sequence
$\langle \min(e^n_\beta\setminus\alpha):n<\omega\rangle$ is decreasing and hence eventually constant.
This brings us to our next definition.

\begin{definition}
We define $\st^*(\alpha,\beta,\ell)$ --- ``step $\ell$ of the settled walk from $\beta$ to $\alpha$ (along $\bar{e}$)'' ---
by the following recursion:
\begin{equation*}
\st^*(\alpha,\beta,0)=\beta,
\end{equation*}
and
\begin{equation*}
\st^*(\alpha,\beta,\ell+1)=
\begin{cases}
\alpha &\text{if $\alpha=\st^*(\alpha,\beta,\ell)$,}\\
&\\
\lim_{n\rightarrow\infty}(\min(e^n_{\st^*(\alpha,\beta,\ell)}\setminus\alpha) &\text{otherwise.}
\end{cases}
\end{equation*}
We let $n^*(\alpha,\beta)$ denote the least $n$ for which $\st^*(\alpha,\beta, n)=\alpha$.
\end{definition}

The settled walks described above avoid the use of parameters $s$; unfortunately, we seem to need the greater
generality furnished by Definition~\ref{defn1} in our proof of the main result of this paper.
The following straightforward lemma connects the two concepts.

\begin{lemma}
\label{settlelemma}
There is an $m^*<\omega$ such that if $s\in^{\,<\omega}\!\omega$, $\lg(s)\geq n^*(\alpha,\beta)$,
and $s(i)\geq m^*$ for all $i<\lg(s)$, then
\begin{equation*}
\st(\alpha,\beta,s,\ell)=\st^*(\alpha,\beta,\ell)\text{ for all }\ell<n^*(\alpha,\beta).
\end{equation*}
\end{lemma}

\noindent We say that $m^*$ {\em settles the walk from $\beta$ to
$\alpha$ (along $\bar{e}$)}, and let $m^*(\alpha,\beta)$ denote the least such $m^*$.

Our discussion now returns to a familiar context --- let $\lambda=\mu^+$ for $\mu$ singular of countable cofinality, and let $S$ be a stationary
subset of $\{\delta<\lambda:\cf(\delta)=\aleph_0\}$.  Further suppose $(\bar{C},\bar{I})$ is a
well-formed $S$-club system swallowed by the generalized $C$-sequence $\bar{e}$.  In the course of this
discussion, we will define several auxiliary functions.

Suppose $\delta\in S$ and $\delta<\beta<\lambda$, and let $m^*=m^*(\delta,\beta)$  be as in Lemma~\ref{settlelemma}.
For
$\ell<n^*(\delta,\beta)-1$, we know $\delta\notin
e^{m^*}_{\st^*(\delta,\beta,\ell)}$ and so if we define
\begin{equation}
\label{gamma*def}
\gamma^*=\gamma^*(\delta,\beta)=\sup\{\max(e^{m^*}_{\st^*(\delta,\beta,\ell)}\cap\delta):\ell<n^*(\delta,\beta)-1\},
\end{equation}
then $\gamma^*$ must be less than $\delta$.

Let $\gamma=\gamma(\delta,\beta)$ denote the ordinal $\st^*(\delta,\beta,
n^*(\delta,\beta)-1$);  our choice of $m^*$ ensures that $\delta$
is in $S\cap e^{m^*}_\gamma$.  An appeal to Lemma~\ref{swallowlemma} tells us there
must exist a least $\bar{m}=\bar{m}(\delta,\beta)<\omega$
such that
\begin{enumerate}
\item $\bar{m}\geq m^*$,
\medskip
\item $\nacc(C_\delta)\cap I_\delta(m)\subs \nacc(e^m_\gamma)$ for all $m\geq
\bar{m}$, and
\medskip
\item if $m\geq \bar{m}$ and $\beta^*\in \nacc(C_\delta)\cap I_\delta(m)$, then
\begin{equation}
\gamma^*<\sup(e^m_\gamma\cap\beta^*)<\beta^*.
\end{equation}
\end{enumerate}

\begin{definition}
Suppose $\delta\in S$, and $\delta<\beta<\lambda$.
For each $m<\omega$, we let
$s(\delta,\beta,m)\in^{\,\omega}\!\omega$ be the sequence of length $n^*(\delta,\beta)$
defined by
\begin{equation*}
s(\delta,\beta,m)[\ell]=
\begin{cases}
m^*(\delta,\beta) &\text{if $\ell<n^*(\delta,\beta)-1$,}\\
m  &\text{if $\ell=n^*(\delta,\beta)-1$}.
\end{cases}
\end{equation*}
\end{definition}

\begin{proposition}
\label{claim3.7}
Suppose $\delta\in S$, $\delta<\beta<\lambda$, and $m\geq\bar{m}(\delta,\beta)$.
For any $\beta^*\in \nacc(C_\delta)\cap I_\delta(m)$, if $\sup(e^m_{\gamma(\delta,\beta)}\cap\beta^*)<\alpha<\beta^*$,
then
\begin{equation}
\st(\alpha,\beta,s(\delta,\beta, m),\ell)=\st^*(\delta,\beta,\ell)\text{ for all
}\ell<n^*(\delta,\beta),
\end{equation}
and
\begin{equation}
\st(\alpha,\beta,s(\delta,\beta, m), n^*(\beta,\delta))=\beta^*.
\end{equation}
\end{proposition}
\begin{proof}
Assume $\alpha$ and $s:=s(\delta,\beta,m)$ are as hypothesized, and suppose
\begin{equation*}
\st(\alpha,\beta,s,\ell)=\st^*(\alpha,\beta,\ell)
\end{equation*}
with $\ell+1<n^*(\delta,\beta)$. Then
\begin{eqnarray*}
\st(\alpha,\beta,s,\ell+1)& = &\min(e^{s(\ell)}_{\st(\alpha,\beta,s,\ell)}\setminus\alpha)\\
   & = & \min(e^{m^*}_{\st^*(\delta,\beta,\ell)}\setminus\alpha)\\
   & = & \min(e^{m^*}_{\st^*(\delta,\beta,\ell)}\setminus\delta)\,\,\text{(as $\alpha>\gamma^*(\delta,\beta)$)}\\
   & = & \st^*(\delta,\beta,\ell+1).
\end{eqnarray*}
In particular, we know
\begin{equation*}
\st(\alpha,\beta,s,n^*(\delta,\beta)-1)=\st^*(\delta,\beta,n^*(\delta,\beta)-1)=\gamma(\delta,\beta).
\end{equation*}

We now use Definition~\ref{defn1} to compute
\begin{eqnarray*}
\st(\alpha,\beta,s, n^*(\delta,\beta)) & = &
\min(e^{s(n^*(\delta,\beta)-1)}_{\st^*(\alpha,\beta,s, n^*(\delta,\beta)-1)}\setminus \alpha) \\
   & = & \min(e^m_{\gamma(\delta,\beta)}\setminus\alpha)\\
   & = & \beta^*,
   \end{eqnarray*}
where the last equality holds because $\beta^*\in e^m_{\gamma(\delta,\beta)}$ and
\begin{equation*}
\sup(e^m_{\gamma(\delta,\beta)}\cap\beta^*)<\alpha<\beta^*.
\end{equation*}
\end{proof}

The preceding argument certainly benefits from a description in English. Given $\delta<\beta$ with $\delta\in S$,
if we define $\gamma^*$ as in (\ref{gamma*def}), then the usual sort of minimal walks argument guarantees
that for any $\alpha$ in the interval $(\gamma^*,\delta)$, the ``$m^*$--walk'' (i.e., the walk obtained by
always stepping in the $m^*$th ladder) from $\beta$ to $\alpha$ will agree with the $m^*$--walk from $\beta$ to
$\delta$ until the last step before the latter arrives at $\delta$.  Varying the ladder used for the next step
(i.e., changing the particular value of $m$) gives us a way
of gaining control over one more step, provided we have a little more information on the ordinal $\alpha$.

Notice that even though we assume $m\geq m^*$,
we cannot simply replace $s(\delta,\beta, m)$ with a sequence of the same length that is constant with
value $m$ --- doing this change has no effect on our steps in the initial portion of the walk, but it might increase
the value of $\gamma^*$ so that it exceeds the particular $\beta^*$ we were aiming for,
 and then the argument no longer works (although something could be said if we were working with $\delta$ of uncountable cofinality).
  Thus, we seem to be stuck with sequences $s$ that are not constant if we want our proof to go through.

\section{The main theorem}

Throughout this section, we will be operating in the following general context:

\begin{itemize}
\item $\lambda=\mu^+$ for $\mu$ singular of cofinality $\aleph_0$
\medskip
\item $S$ is a stationary subset of $\{\delta<\lambda:\cf(\delta)=\aleph_0\}$
\medskip
\item $(\bar{C},\bar{I})$ is a well-formed $S$-club system
\medskip
\item $\bar{e}=\langle e^n_\alpha:n<\omega,\alpha<\lambda\rangle$ is a generalized $C$-sequence that swallows $(\bar{C},\bar{I})$
\medskip
\item $(\vec{\mu},\vec{f})$ is a scale for $\mu$ with $\mu_0>\aleph_0$.
\medskip
\item $\Gamma:[\lambda]^2\rightarrow\omega$ is the function defined (for $\alpha<\beta$) by
\begin{equation}
\Gamma(\alpha,\beta) = \max\{i<\omega: f_\beta(i)\leq f_\alpha(i)\}.
\end{equation}
\item $\langle s_i:i<\omega\rangle$ is an enumeration of $^{<\omega}\!\omega$ in which each element appears infinitely often
\medskip
\item $x=\{\lambda,\mu,S,(\bar{C},\bar{I}),\bar{e},(\vec{\mu},\vec{f}), \langle s_i:i<\omega\rangle\}$ (so $x$
codes all of the parameters listed previously)
\medskip
\item $\mathfrak{A}$ is a structure of the form $\langle H(\chi), \in, <_\chi\rangle$ for some sufficiently large
regular cardinal~$\chi$ and well-ordering $<_\chi$ of $H(\xi)$.
\end{itemize}

We apologize to the reader for the preceding bare list of assumptions --- writing all of the above out results in
a dramatic loss of clarity.

\begin{definition}
\label{colordef}
We define a coloring $c:[\lambda]^2\rightarrow\lambda$ as follows:
\vskip.2in
\noindent For $\alpha<\beta<\lambda$, let
\begin{equation}
s^*(\alpha,\beta)=s_{\Gamma(\alpha,\beta)}
\end{equation}
Next, define
\begin{equation}
k(\alpha,\beta) = \text{least $\ell\leq n(\alpha,\beta)$ such that $\Gamma(\alpha,\st(\alpha,\beta,s^*(\alpha,\beta),\ell))\neq\Gamma(\alpha,\beta)$.}
\end{equation}
Finally, let
\begin{equation}
c(\alpha,\beta)=\st(\alpha,\beta, s^*(\alpha,\beta), k(\alpha,\beta)).
\end{equation}
\end{definition}

The computation of $c(\alpha,\beta)$ seems more reasonable when written out in English --- we start
by computing $\Gamma(\alpha,\beta)$ and use this to select the element $s^*$ of $^{<\omega}\!\omega$
that will guide our walk.  We then walk from $\beta$ to $\alpha$ using $s^*$, and we stop when we reach
a point where ``$\Gamma$ changes''.  This stopping point is the value of $c(\alpha,\beta)$.
The same basic idea is exploited in~\cite{nsbpr}; the current version is complicated by our need for
the parameter $s^*$.

\begin{theorem}
\label{mainthm}
If $\langle t_\alpha:\alpha<\lambda\rangle$ is a pairwise disjoint sequence of finite subsets of $\lambda$
and $A$ is an unbounded subset of $\lambda$, then for $\id_p(\bar{C},\bar{I})$-almost all $\beta^*<\lambda$,
we can find $\alpha<\lambda$ and $\beta\in A$ such that
\begin{equation}
c(\epsilon,\beta) = \beta^*\text{ for all }\epsilon\in t_\alpha.
\end{equation}
\end{theorem}
\begin{proof}
By way of contradiction, suppose $\langle t_\alpha:\alpha<\lambda\rangle$ and $A\subs\lambda$
form a counterexample (without loss of generality, $\alpha<\min(t_\alpha)$).  Then there is an $\id_p(\bar{C},\bar{I})$-positive set $B$ such that
for each $\beta^*\in B$, there are no $\alpha<\lambda$ and $\beta\in A$ such that $c\restr t_\alpha\times\{\beta\}$
is constant with value $\beta^*$.

Let $\langle M_\xi:\xi<\lambda\rangle$ be a $\lambda$-approximating
sequence over $\{x, \langle t_\alpha:\alpha<\lambda\rangle, A\}$, and let
$E$ be the closed unbounded set defined by
\begin{equation*}
E:=\{\delta<\lambda: \delta=M_\delta\cap\lambda\}.
\end{equation*}
By our assumptions, we can choose $\delta\in E\cap S$ such that
\begin{equation}
E\cap B\cap C_\delta\notin I_\delta.
\end{equation}
Finally, let $\beta$ be some element of $A$ greater than $\delta$.

The discussion preceding Proposition~\ref{claim3.7} applies to $\delta$ and $\beta$,
so we can safely speak of $\bar{m}(\delta,\beta)$ and the other functions defined there.
Since $E\cap B\cap C_\delta\notin I_\delta$,  we know that $E\cap B$ must contain members
 of $\nacc(C_\delta)\cap I_\delta(n)$ for arbitrarily large~$n$. Thus,
we can find $\beta^*\in E\cap B$ such that $\beta^*\in \nacc(C_\delta)\cap I_\delta(m)$
for some $m\geq \bar{m}(\delta,\beta)$. In particular,
\begin{equation}
\beta^*\in\nacc(e^m_{\gamma(\delta,\beta)})
\end{equation}
by the definition $\bar{m}(\delta,\beta)$.

Let $s^*= s(\delta,\beta, m)$ for this particular value of $m$. We know
\begin{equation}
\label{eqn15}
\sup(e^m_{\gamma(\delta,\beta)}\cap\beta^*)<\beta^*,
\end{equation}
and Proposition~\ref{claim3.7} can now be brought into play --- if $\epsilon$ lies
in the interval determined by~(\ref{eqn15}), then we know that the $s^*$-walk
from $\beta$ to $\epsilon$ will pass through $\beta^*$ and in addition, we know
exactly what the walk looks like up to that point.

Since $\beta^*\in E$ and $\langle t_\alpha:\alpha<\lambda\rangle\in M_0$, we note
\begin{equation}
\alpha<\beta^*\Longrightarrow t_\alpha\subs\beta^*.
\end{equation}
We assumed $\alpha<\min(t_\alpha)$, and so we conclude
\begin{equation}
\sup(e^m_{\gamma(\delta,\beta)}\cap\beta^*)<\alpha<\beta^* \Longrightarrow t_\alpha\subs (\sup(e^m_{\gamma(\delta,\beta)}\cap\beta^*),\beta^*).
\end{equation}
We now prove the following claim.

\begin{claim}
For all sufficiently large $i<\omega$, there are unboundedly many $\alpha<\beta^*$ such that
\begin{equation}
\label{eqn1}
\Gamma(\epsilon,\st(\alpha,\beta, s^*,\ell))=i\text{ for all }\ell<n^*(\delta,\beta)\text{ and $\epsilon\in t_\alpha$,}
\end{equation}
while
\begin{equation}
\label{eqn2}
\Gamma(\epsilon,\beta^*)>i\text{ for all }\epsilon\in t_\alpha.
\end{equation}
\end{claim}
\begin{proof}
Let $M$ be the Skolem hull (in $\mathfrak{A}$) of $\{x, \langle t_\alpha:\alpha<\lambda\rangle, A, \beta^*\}$.
Since $M$ is countable and the $\mu_i$ are uncountable, it follows that
\begin{equation}
\Ch_M(i)=\sup(M\cap\mu_i)\text{ for all }i<\omega,
\end{equation}
where $\Ch_M$ is the characteristic function of $M$ from Definition~\ref{chardef}.

For each $\alpha<\lambda$, let $f^{\min}_\alpha$ be the function with domain $\omega$ defined as
\begin{equation}
f^{\min}_\alpha(i)= \min\{f_\epsilon(i):\epsilon\in t_\alpha\}.
\end{equation}
It is easy to see that $(\vec{\mu},\langle f_\alpha^{\min}:\alpha<\lambda\rangle)$ is a
scale for $\mu$,
and this scale is also an element of $M_{\beta^*}$.  Since $\beta^*$ is an element of every
closed unbounded subset of $\lambda$ that is an element of $M_{\beta^*}$, we can appeal to Lemma~\ref{scalelemma1}
and conclude that there is an $i_0<\omega$ such that whenever $i_0\leq i<\omega$,
\begin{equation}
\label{eqn4}
(\forall\eta<\mu_i)(\forall\nu<\mu_{i+1})(\exists^*\alpha<\beta^*)[f^{\min}_\alpha(i)>\eta\wedge f^{\min}_\alpha(i+1)>\nu.]
\end{equation}

Next, note that $M$ is an element of $M_\delta$, as the required Skolem hull  can be computed in
$M_\delta$ using the model $M_{\beta^*+1}$.
This means that the function $\Ch_M$ is in $M_\delta$ and therefore
\begin{equation}
\Ch_M <^* f_\delta.
\end{equation}
Thus, we can  find $i_1<\omega$ such that
\begin{equation}
\label{eqn8}
\Ch_M\restr [i_1,\omega)< f_{\st(\delta,\beta,s^*,\ell)}\restr [i_1,\omega)\text{ for all }\ell<n^*(\delta,\beta).
\end{equation}
Finally, choose $i_2$ so large that
\begin{equation}
\label{eqn3}
\cf(\beta^*)<\mu_{i_2},
\end{equation}
and let $i^*=\max\{i_0, i_1, i_2\}$.

We claim now that (\ref{eqn1}) and (\ref{eqn2}) holds for any $i\geq i^*$. Given such an~$i$,
we define
\begin{eqnarray*}
N &= &\Sk_{\mathfrak{A}}(M\cup\mu_i)\\
\eta &= &\sup\{f_{\st(\delta,\beta,s^*,\ell)}(i):\ell<n^*(\delta,\beta)\}, \text{ and }\\
\nu &= &f_{\beta^*}(i+1).
\end{eqnarray*}

We know~(\ref{eqn4}) holds in the model~$N$, and since both $\eta$ and $\nu$ (defined above) are
in the model $N$ (as $\mu_i\subs N$ and $f_{\beta^*}\in N$), it follows that
\begin{equation}
N\models (\exists^*\alpha<\beta^*)[f^{\min}_\alpha(i)>\eta\wedge f^{\min}_\alpha(i+1)>\nu]
\end{equation}
The definition of $N$ together with (\ref{eqn3}) imply that $N\cap\beta^*$ is unbounded in~$\beta^*$, and so
 we can conclude that the set of $\alpha\in N\cap\beta^*$ for which
\begin{equation}
\label{eqn5}
f^{\min}_\alpha(i)>\eta\text{ and } f^{\min}_\alpha(i+1)>\nu
\end{equation}
is unbounded in $\beta^*$.

Suppose now that $\alpha<\beta^*$ satisfies~(\ref{eqn5}). If in addition $\sup(e^m_{\gamma(\delta,\beta)}\cap\beta^*)<\alpha$, then given
$\epsilon\in t_\alpha$, we know
\begin{equation}
\sup(e^m_{\gamma(\delta,\beta)}\cap\beta^*)<\alpha\leq \epsilon <\beta^*.
\end{equation}
An appeal to Proposition~\ref{claim3.7} tells us
\begin{equation}
\st(\epsilon,\beta, s^*, \ell)=\st^*(\delta,\beta,\ell)\text{ for all }\ell<n^*(\delta,\beta),
\end{equation}
and
\begin{equation}
\st(\epsilon,\beta, s^*, n^*(\beta, \delta))=\beta^*.
\end{equation}

Now it should be clear that $\Gamma(\epsilon,\beta^*)\geq i+1$ because of our choice of $\nu$.
Given $\ell<n^*(\beta,\delta)$, we know
\begin{equation}
f_{\st(\epsilon,\beta,s^*,\ell)}(i)=f_{\st^*(\delta,\beta,\ell)}(i)\leq\eta < f_\alpha^{\min(i)}\leq f_\epsilon(i).
\end{equation}
On the other hand, given $j>i$ we know (from Lemma~\ref{skolemhulllemma}) that
\begin{equation}
\Ch_M(j)=\Ch_N(j)=\sup(N\cap\mu_j),
\end{equation}
and since $\epsilon\in N$ (as $\epsilon\in t_\alpha\in N$ and $t_\alpha$ is finite), it follows from~(\ref{eqn8}) that
\begin{equation}
f_\epsilon(j)\leq\Ch_N(j)=\Ch_M(j)<f_{\st^*(\delta,\beta,\ell)}(j)=f_{\st(\epsilon,\beta,s^*,\ell)}(j)
\end{equation}
for all $\ell<n^*(\delta,\beta)$. The statement~(\ref{eqn1}) now follows immediately and with it the claim.
\end{proof}

We are now in a position to obtain a contradiction. First, use the preceding claim
to fix an $\bar{i}$ such that
such that
\begin{equation}
s_{\bar{i}}= s^*,
\end{equation}
and for which there are unboundedly many $\alpha\leq\beta^*$ satisfying both~(\ref{eqn1}) and~(\ref{eqn2}).
In particular, we can fix an $\alpha<\beta^*$ in $A$ satisfying~(\ref{eqn1}) and~(\ref{eqn2}) such that
\begin{equation}
\sup(e^m_{\gamma(\delta,\beta)}\cap\beta^*)<\alpha<\beta^*;
\end{equation}
we now prove
\begin{equation}
c(\epsilon,\beta)=\beta^*\text{ for all $\epsilon\in t_\alpha$},
\end{equation}
and this will yield the desired contradiction.

Given $\epsilon\in t_\alpha$, from (\ref{eqn1}), we conclude $\Gamma(\epsilon,\beta)=\bar{i}$, and hence
\begin{equation}
s^*(\epsilon,\beta)=s^*.
\end{equation}
For $\ell<n^*(\delta,\beta)$, we know
\begin{equation}
\Gamma(\epsilon,\st(\epsilon,\beta, s(\epsilon,\beta),\ell))=\Gamma(\epsilon,\st^*(\delta,\beta,\ell))=\bar{i}=\Gamma(\epsilon,\beta),
\end{equation}
while
\begin{equation}
\Gamma(\epsilon,\st(\alpha,\beta, s(\epsilon,\beta),n^*(\delta,\beta)))=\Gamma(\epsilon,\beta^*)>\bar{i}.
\end{equation}
Thus
\begin{equation}
k(\epsilon,\beta)=n^*(\delta,\beta),
\end{equation}
and
\begin{equation}
c(\epsilon,\beta)=\st(\epsilon,\beta,s^*(\epsilon,\beta), k(\epsilon,\beta))=\st(\epsilon,\beta, s^*, n^*(\delta,\beta))=\beta^*,
\end{equation}
as required.

The contradiction is immediate as no such $\alpha$ and $\beta$ are supposed to exist for our choice of~$\beta^*$.
\end{proof}

\section{Conclusions}

We now use Theorem~\ref{mainthm} to draw some conclusions concerning negative square-brackets partition relations and
their connection with saturation-type properties of club-guessing ideals.
These results are framed in terms of successors of singular cardinals of countable cofinality because stronger results
are known for the uncountable cofinality case (see \cite{Sh:365},  \cite{535}, and \cite{nsbpr}).  These results are
also weaker than those claimed for the countable cofinality case in Section~4 of~\cite{Sh:365} --- as mentioned before, there is a problem in the proof of Lemma~4.2(4) on page~162; the present paper provides a partial rescue.

Let us recall the following definitions:

\begin{definition}
Let $I$ be an ideal on some set $A$, and let $\sigma$ and $\tau$ be cardinals, with $\tau$ regular.
\begin{enumerate}
\item The ideal $I$ is weakly $\sigma$-saturated if $A$ cannot be partitioned into $\sigma$ disjoint $I$-positive
sets, i.e., there is no function $\pi:A\rightarrow\sigma$ such that
\begin{equation*}
\pi^{-1}(i)\notin I
\end{equation*}
for all $i<\sigma$.
\sk
\item The ideal $I$ is $\tau$-indecomposable if $\bigcup_{i<\tau}A_i\in I$ whenever $\langle A_i:i<\tau\rangle$ is an {\em increasing} sequence of sets from $I$.
\end{enumerate}
\end{definition}

\begin{theorem}
\label{thm9}
Suppose $\lambda=\mu^+$ for $\mu$ singular of countable cofinality, and let $\theta\leq\lambda$.
If there is a well-formed pair $(\bar{C},\bar{I})$ for which the ideal $\id_p(\bar{C},\bar{I})$ fails to be weakly $\theta$-saturated, then there is a coloring $c^*:[\lambda]^2\rightarrow \theta$ such that for any two unbounded subsets $A$ and $B$ of $\lambda$ and any $\varsigma<\theta$, there are $\alpha\in A$ and $\beta\in B$ with $\alpha<\beta$ and
\begin{equation}
c^*(\alpha,\beta)=\varsigma.
\end{equation}
In particular, $\lambda\nrightarrow[\lambda]^2_\theta$.
\end{theorem}
\begin{proof}
Suppose there is a function $\pi:\lambda\rightarrow\kappa$ such that $\pi^{-1}(\{\epsilon\})$ is $\id_p(\bar{C},\bar{I})$-positive
for each $\epsilon<\kappa$.  Define the function $c^*:[\lambda]^2\rightarrow\kappa$ by
\begin{equation}
c^*(\alpha,\beta)=\pi(c(\alpha,\beta)).
\end{equation}
Given $A$ and $B$ unbounded in $\lambda$ and $\varsigma<\kappa$, since $\pi^{-1}(\{\varsigma\})$ is $\id_p(\bar{C},\bar{I})$-positive
 we can apply Theorem~\ref{mainthm} (with $\langle\{\alpha\}:\alpha\in A\rangle$ in place of $\langle t_\alpha:\alpha<\lambda\rangle$)
to find $\alpha\in A$ and $\beta\in B$ such that
\begin{equation}
c(\alpha,\beta)\in\pi^{-1}(\{\varsigma\}),
\end{equation}
and this suffices.
\end{proof}

We state the following corollary in such a way that it covers all successors of singular cardinals, though
we remind the reader that stronger results are known (see \cite{nsbpr}) in the situation where the cofinality
of $\mu$ is uncountable.

\begin{corollary}
Let $\mu$ be a singular cardinal. If $\mu^+\rightarrow[\mu^+]^2_{\mu^+}$,
then there is an ideal $I$ on $\mu^+$ such that
\begin{enumerate}
\item $I$ is a proper ideal extending the non-stationary ideal on $\mu^+$,
\medskip
\item $I$ is $\cf(\mu)$-complete
\medskip
\item $I$ is $\tau$-indecomposable for all uncountable regular $\tau$ with $\cf(\mu)<\tau<\mu$, and
\medskip
\item $I$ is weakly $\theta$-saturated for some $\theta<\mu$.
\end{enumerate}
\end{corollary}
\begin{proof}
Let $S$ be any stationary subset of $\{\delta<\mu^+:\cf(\delta)=\cf(\mu)\}$, and let $(\bar{C},\bar{I})$
be a well-formed (or nice in the case where $\cf(\mu)>\aleph_0$) $S$-club system. An elementary argument tells us that $\mu^+\rightarrow [\mu^+]^2_\mu$ must hold, and therefore the ideal $\id_p(\bar{C},\bar{I})$ is weakly $\mu$-saturated --- this follows from
Theorem~\ref{thm9} in the case where $\cf(\mu)=\aleph_0$, and Theorem~3 of~\cite{nsbpr} if $\cf(\mu)>\aleph_0$.  It is also routine to check (see Observation 3.2(1) on page~139 of~\cite{cardarith}) that $\id_p(\bar{C},\bar{I})$ satisfies conditions (1)-(3).

Now if $\id_p(\bar{C},\bar{I})$ happens to be weakly $\cf(\mu)$-saturated (a situation which might not even be consistent --- see Section~6 of~\cite{nsbpr}) then we are done. Otherwise, we can find a family $\{A_i:i<\cf(\mu)\}$ of  disjoint $\id_p(\bar{C},\bar{I})$-positive sets. Since $\id_p(\bar{C},\bar{I})$ is weakly $\mu$-saturated, there must exist an $i<\cf(\mu)$ and a $\theta<\mu$ such that $A_i$
cannot be partitioned into $\theta$ disjoint $\id_p(\bar{C},\bar{I})$-positive sets.  If we define
\begin{equation*}
I:=\id_p(\bar{C},\bar{I})\restr A_i := \{B\subs\mu^+: A_i\cap B\in \id_p(\bar{C},\bar{I})\},
\end{equation*}
then $I$ has all of the required properties.
\end{proof}

\bibliographystyle{plain}

\begin{thebibliography}{10}

\bibitem{jb}
James E. Baumgarter.
\newblock A new class of order-types.
\newblock {\em Ann. Pure Appl. Logic}, 54(3):195--227, 1991.



\bibitem{535}
T.~Eisworth and S.~Shelah.
\newblock Successors of singular cardinals and coloring theorems {I}.
\newblock {\em Arch. Math. Logic}, 44(5):597--618, 2005.


\bibitem{nsbpr}
Todd Eisworth
\newblock A note on strong negative partition relations.
\newblock to appear in Fund. Math.

\bibitem{myhandbook}
Todd Eisworth
\newblock Successors of singular cardinals.
\newblock (chapter in the forthcoming {\em Handbook of Set Theory}).

\bibitem{ehr}
P.~Erd{\H{o}}s, A.~Hajnal, and R.~Rado.
\newblock Partition relations for cardinal numbers.
\newblock {\em Acta Math. Acad. Sci. Hungar.}, 16:93--196, 1965.

\bibitem{abc}
Menachem Kojman.
\newblock The {A}{B}{C} of {P}{C}{F}.
\newblock unpublished manuscript.

\bibitem{cardarith}
Saharon Shelah.
\newblock {\em {Cardinal Arithmetic}}, volume~29 of {\em {Oxford Logic
  Guides}}. {Oxford University Press}, 1994.


\bibitem{Sh:365}
Saharon Shelah.
\newblock {There are Jonsson algebras in many inaccessible cardinals}.
\newblock In {\em {Cardinal Arithmetic}}, volume~29 of {\em {Oxford Logic
  Guides}}, Chapter III. {Oxford University Press}, 1994.



\bibitem{stevochapter}
Stevo Todor{\v{c}}evi{\'c}.
\newblock Coherent sequences.
\newblock Chapter in the forthcoming Handbook of Set Theory.

\bibitem{minimal}
Stevo Todor{\v{c}}evi{\'c}.
\newblock Partitioning pairs of countable ordinals.
\newblock {\em Acta Math.}, 159(3-4):261--294, 1987.

\bibitem{stevobook}
Stevo Todor{\v{c}}evi{\'c}.
\newblock {\em {Walks on ordinals and their characteristics}}, volume~263 of {\em {Progress in Mathematics}}. {Birkhauser}, 2007.


\end{thebibliography}

\end{document}